\documentclass[reqno,11pt]{amsart}

\usepackage{xcolor}
\usepackage{graphicx}
\usepackage[hidelinks]{hyperref}
\usepackage{amscd}

\usepackage[T1]{fontenc}
\usepackage[utf8]{inputenc}
\usepackage{amsmath}
\usepackage{amssymb, amsmath}
\usepackage{amsthm,enumerate}
\usepackage{amssymb}
\usepackage{latexsym,array}
\usepackage{amsfonts}
\usepackage{shadow}
\usepackage{amsbsy}
\usepackage{amssymb}
\usepackage{dsfont}
\usepackage{mathtools}
\usepackage{enumitem}
\usepackage[utf8]{inputenc}
\usepackage[notref, notcite, final]{showkeys}
\usepackage{color}
\usepackage{graphicx}
\usepackage[hidelinks]{hyperref}
\usepackage{tikz}
\usepackage{float}
\usepackage{dsfont}

\numberwithin{equation}{section}

\newtheorem{theorem}{Theorem}[section]
\newtheorem{lemma}[theorem]{Lemma}
\newtheorem{proposition}[theorem]{Proposition}
\theoremstyle{definition}
\newtheorem{remark}[theorem]{{\bf Remark}}
\newtheorem{definition}[theorem]{Definition}

\title[]{Characterization of continuous homomorphisms \\ on entire slice monogenic functions}

\author[S. Pinton]{Stefano Pinton}
\address{\hspace{-0.5cm}(SP) Politecnico di Milano, Dipartimento di Matematica, Via E. Bonardi 9, 20133 Milano, Italy}
\email{stefano.pinton@polimi.it}

\author[P. Schlosser]{Peter Schlosser}
\address{\hspace{-0.5cm}(PS) Politecnico di Milano, Dipartimento di Matematica, Via E. Bonardi, 9, 20133 Milano, Italy}
\email{pschlosser@math.tugraz.at}

\begin{document}

\begin{abstract}
This paper is inspired by a class of infinite order differential operators arising in quantum mechanics. They turned out to be an important tool in the investigation of evolution of superoscillations with respect to quantum fields equations. Infinite order differential operators act naturally on spaces of holomorphic functions or on hyperfunctions. Recently, infinite order differential operators have been considered and characterized on the spaces of entire monogenic functions, i.e., functions that are in the kernel of the Dirac operators. The focus of this paper is the characterization of infinite order differential operators that act continuously on a different class of hyperholomorphic functions, called slice hyperholomorphic functions with values in a Clifford algebra or also slice monogenic functions. This function theory has a very reach associated spectral theory and both the function theory and the operator theory in this setting are subjected to intensive investigations. Here we introduce the concept of proximate order and establish some fundamental properties of entire slice monogenic functions that are crucial for the characterization of infinite order differential operators acting on entire slice monogenic functions.
\end{abstract}

\maketitle


\vspace{1cm}

\textbf{Acknowledgements:} S. Pinton is supported by MUR grant "Dipartimento di Eccellenza 2023-2027". P. Schlosser was funded by the Austrian Science Fund (FWF) under Grant No. J 4685-N and by the European Union -- NextGenerationEU.

\tableofcontents

\section{Introduction}\label{INTROD1}

Extensive research has been conducted on infinite order differential operators for a considerable period of time. In recent years, their significance in examining the evolution of superoscillations as initial data for the Schrödinger equation has emerged as a fundamental area of study. Superoscillatory functions arise in various fields of science and technology. In quantum mechanics, they are the result of Aharonov's weak values, as outlined in \cite{aav,abook}. Investigating their time evolution as initial data for quantum field equations represents a crucial problem in the domain of quantum mechanics. Further details can be found in the monograph \cite{acsst5}, as well as in \cite{ABCS19,acsst3,acsst6,KGField,Jussi,TRANS,genHYP,genHYPAA,PETER}. Moreover, we point out that superoscillations have been studied not only in quantum mechanics, but also in various fields including optics, antenna theory and signal processing. In particular, it is believed that they can be used to improve the resolution of imaging and spectral analysis techniques, as they allow for obtaining detailed information about structures at smaller scales than those allowed by conventional waves. Superoscillations are still an active area of study, and there are still many open questions about their nature and potential applications.

\medskip

Investigating the evolution of superoscillatory functions under the time dependent Schrödinger equation presents highly intricate problems. A natural functional analytic framework for this purpose is the space of entire functions with specific growth conditions. In fact, addressing the Cauchy problem for the Schrödinger equation with superoscillatory initial data, we are naturally
led to study infinite order differential operators.

\medskip

In order to explain how infinite order differential operators appear in quantum mechanics it is enough to consider the simple case of the free particle for Schrödinger equation with initial datum that is a superoscillatory function. Precisely, we consider the Cauchy problem
\begin{equation}\label{CPSR}
\begin{split}
i\frac{\partial\psi(x,t)}{\partial t}&=-\frac{\partial^2\psi(x,t)}{\partial x^2} \\
\psi(x,0)&=F_n(x,a):=\sum_{k=0}^n{n\choose k}\big(\frac{1+a}{2}\big)^{n-k}\big(\frac{1-a}{2}\big)^ke^{i(1-2k/n)x}
\end{split}
\end{equation}
for $a\in\mathbb{R}$, $a>1$, $x\in\mathbb{R}$ where $F_n(x,a)$ is a superoscillatory function outcome of Aharonov's weak values. In \cite{acsst3} it has been shown that the solution $\psi_n(x,t)$, can be written as:
\begin{equation*}
\psi_n(x,t)=\sum_{m=0}^\infty\frac{(it)^m}{m!}\frac{d^{2m}}{dx^{2m}}F_n(x,a)
\end{equation*}
for every $x\in\mathbb{R}$ and $t\in\mathbb{R}$ and to study the limit, as $n\to\infty$ for $\psi_n(x,t)$ entails the investigation of persistence of superoscillation during the time evolution. The crucial fact is to show the continuity of the operator
\begin{equation*}
U\Big(\frac{d}{dx}\Big):=\sum_{m=0}^\infty\frac{(it)^m}{m!}\frac{d^{2m}}{dx^{2m}}
\end{equation*}
on a class of functions that contains $F_n(x,a)$. In order to study the continuity properly, we recall that we have to extend both the operator $U(\frac{d}{dx})$ and the function $F_n(x,a)$ to the complex setting replacing the real variable $x$ by the complex variable $z$. The natural settings for this problem are the spaces $\mathcal{A}_p$, for $p\geq 1$, i.e., the space of entire functions with either order lower than $p$ or order equal to $p$ and finite type, i.e., it consists of functions $f$ for which there exist constants $B,C>0$ such that
\begin{equation}\label{ABC}
|f(z)|\leq Ce^{B|z|^p}.
\end{equation}
The convergence in these spaces is defined as follows. Let $(f_n)_{n\in\mathbb{N}},f_0\in\mathcal{A}_p$. Then $f_n\to f_0$ in $\mathcal{A}_p$ if there exists some $B>0$ such that
\begin{equation}
\lim\limits_{n\to\infty}\sup_{z\in\mathbb{C}}\Big|(f_n(z)-f_0(z))e^{-B|z|^p}\Big|=0.
\end{equation}
Proving the continuity of the infinite order differential operator $U(\frac{d}{dz})$ on $\mathcal{A}_1$ implies that
\begin{equation*}
\lim_{n\to\infty}U\Big(\frac{d}{dz}\Big)F_n(z,a)=U\Big(\frac{d}{dz}\Big)\lim_{n\to\infty}F_n(z,a).
\end{equation*}
Thus we have $\lim_{n\to\infty}\psi_n(z,t)=e^{iaz-ia^2t}$ in $\mathcal{A}_1$. Taking the restriction to the real axis for the complex variable $z$, we have $\lim_{n\to\infty}\psi_n(x,t)=e^{iax-ia^2t}$ uniformly on the compact sets of $\mathbb{R}$ for the space variable.

\medskip

More generally when we investigate the Schrödinger evolution of superoscillatory functions under different potentials $V$ we reduce the problem to the identification of infinite order differential operators of the type
\begin{equation*}
\mathcal{U}\Big(t,z,\frac{d}{dz}\Big):=\sum_{m=0}^\infty u_n(t,z;V)\frac{d^n}{dz^n}
\end{equation*}
where $z$ is a complex variable, and where the coefficients $u_n(t,z;V)$ depend on the potential $V$ and also on time $t$. The identification and the characterization of these type of operators is crucial in quantum mechanics and from the mathematical point of view it has given a new impulse to the theory of infinite order differential operators acting on holomorphic functions, see \cite{ACSS18,QS2}.

The function theory of the spaces $\mathcal A_p$ from above can also be treated in a much more general framework where the constant order $p$ is replaced by some function $\varrho(|z|)$ satisfying certain properties. In the complex setting these spaces of proximate order are introduced in \cite{V23}. It was then generalized to several complex variables in \cite{LG86}. The differential operator representation of continuous homomorphisms between the spaces of entire functions of given proximate order was only recently proven by T. Aoki and co-authors in \cite{AIO20,Aoki2}. In this paper we will generalize these results for slice monogenic functions acting on a Clifford algebra.

\medskip

In this regard there are two main theories of hyperholomorphic functions the monogenic and the slice monogenic function theory. It is important to mention that for these two classes of hyperholomorphic functions we need to define different types of infinite order differential operators involving suitable products of functions that preserve the required notion of hyperholomorphicity. The investigation of infinite order differential operators started in \cite{ANMATHPHY} and the characterisation continuous homomorphisms for entire monogenic functions of given proximate order is studied in \cite{STAMB}. In the monogenic case entire monogenic functions have been considered by different authors in \cite{BK2012,CDK07,CDK07bis,AlmeidaKra2005,CDK2014,KS2013}, while entire slice hyperholomorphic functions are considered in \cite{ACSBOOK2}.

\medskip

In order to give the perspective of the two classes of hyperholomorphic functions, referring to the notation introduced in the next section we denote by $\mathbb{R}_n$ the  Clifford algebra with imaginary units $e_j$, $j=1,\dots,n$. We recall that (left) monogenic functions, see \cite{BDS82,GHS08}, are those functions $f:U\to\mathbb{R}_n$ continuously differentiable on an open subset $U\subseteq\mathbb{R}^{n+1}$ such that
\begin{equation*}
\mathcal{D}f(x)=0,
\end{equation*}
where $\mathcal{D}$ is the Dirac operator defined by $\mathcal{D}=\partial_{x_0}+\sum_{j=1}^n e_j\partial_{x_j}$. For the case of slice monogenic functions there is not a unique way to define them, but there are different ways to introduce them and the various definitions are not totally equivalent according to the properties on the domains on which they are defined. Precisely, we have the following possible definitions:

\begin{itemize}
\item[(I)] Via the Fueter-Sce-Qian mapping theorem, see Definition \ref{SHolDefMON}.
\item[(II)]  Slice by slice, see \cite{6cssisrael}.
\item[(III)] As functions in the kernel of the global operator introduced in \cite{6Global}.
\item[(IV)] Via the Radon and dual Radon transform, see \cite{6Radon}.
\end{itemize}

The most natural definition  for applications to quaternionic and Clifford operators theory, see \cite{FJBOOK,CGK,ColomboSabadiniStruppa2011}, is the definition appearing in Fueter-Sce-Qian mapping theorem which is well summarized in terms of extensions from holomorphic functions to hyperholomorphic ones. Let $\mathcal{O}(D)$ be the set of holomorphic functions on $D\subseteq\mathbb{C}$ and let $\Omega_D\subseteq\mathbb{R}^{n+1}$ be the rotation of $D$ around the real axis. The first  Fueter-Sce-Qian map $T_\text{FS1}$ applied to $\mathcal{O}(D)$ generates the set $\mathcal{SH}(\Omega_D)$ of slice monogenic functions on $\Omega_D$ (which turn out to be intrinsic) and the second  Fueter-Sce-Qian map $T_\text{FS2}$ applied to $\mathcal{SH}(\Omega_D)$ generates axially monogenic functions on $\Omega_D$. We denote this second class of functions by $\mathcal{M}(\Omega_D)$. The extension procedure is illustrated in the diagram:
\begin{equation*}
\begin{CD}
\textcolor{black}{\mathcal{O}(D)}  @>T_\text{FS1}>> \textcolor{black}{\mathcal{SH}(\Omega_D)}  @>\ \   T_\text{FS2}=\Delta^{(n-1)/2} >>\textcolor{black}{\mathcal{M}(\Omega_D)},
\end{CD}
\end{equation*}
where $T_\text{FS2}=\Delta^{(n-1)/2}$ and $\Delta$ is the Laplace operator in dimension $n+1$. For more details on this important extension procedure from complex to hypercomplex analysis see the book \cite{BOOKSCE}, the references therein and the paper \cite{ANUNO} for related topics.

\medskip

\textit{The plan of the paper.} In Section \ref{SEC2} we have some preliminary results on slice monogenic functions. In Section \ref{Prel} we prove new results on some properties of slice monogenic functions of proximate order. Section \ref{SEC3} is the heart of this paper and contains the characterization representation of continuous homomorphisms on spaces of entire slice monogenic functions.

\section{Preliminary results on slice monogenic functions}\label{SEC2}

In this section we recall basic results on slice monogenic functions (see Chapter 2 in \cite{ACSBOOK2}) and also introduce the notion of proximate order functions (see the book \cite{LG86}). We recall that $\mathbb{R}_n$ is the real Clifford algebra over $n$ imaginary units $e_1,\dots,e_n$, satisfying the commutation relation
\begin{equation*}
e_ie_j=-e_je_i\qquad\text{and}\qquad e_i^2=-1,\qquad i\neq j\in\{1,\dots,n\}.
\end{equation*}
Any \textit{Clifford number} $a\in\mathbb{R}_n$ can uniquely be written as
\begin{equation*}
a=a_0+a_1e_1+\ldots+a_ne_n+\dots+a_{12...n}e_1e_2...e_n=\sum\nolimits_Aa_Ae_A,
\end{equation*}
where the last sum is over all ordered subsets $A=\{i_1,\dots,i_r\}\subseteq\{1,\dots,n\}$ and the basis elements are $e_A:=e_{i_1}\dots e_{i_r}$. Note that for $A=\emptyset$ we set $e_\emptyset:=1$. The \textit{Euclidean norm} of a Clifford number $a\in\mathbb{R}_n$ is
\begin{equation*}
|a|^2:=\sum\nolimits_A|a_A|^2.
\end{equation*}
Any element of the form
\begin{equation*}
x=x_0+\underline{x}=x_0+\sum_{\ell=1}^nx_\ell e_\ell
\end{equation*}
will be called \textit{paravector} and we denote by $\mathbb{R}^{n+1}$ the set of all paravectors. Note that no ambiguity arises since we can identify any $(n+1)$-vector $(x_0,x_1,\dots,x_n)$ of real numbers  with $x_0+e_1x_1+\ldots+e_nx_n$. The {\em conjugate} of a paravector $x$ is given by
\begin{equation*}
\overline{x}:=x_0-\underline{x}=x_0-\sum_{\ell=1}^nx_\ell e_\ell.
\end{equation*}
Recall that $\mathbb{S}$ is the sphere
\begin{equation*}
\mathbb{S}=\{\underline{x}=e_1x_1+\ldots +e_nx_n \;|\;x_1^2+\dots+x_n^2=1\},
\end{equation*}
where every $j\in\mathbb{S}$ satisfies $j^2=-1$. Given an element $x=x_0+\underline{x}\in\mathbb{R}^{n+1}$ let us define
\begin{equation*}
[x]:=\{y\in\mathbb{R}^{n+1}\;|\;y=x_0+j|\underline{x}|,\,j\in\mathbb{S}\},
\end{equation*}
as the $(n-1)$-dimensional sphere in $\mathbb{R}^{n+1}$. For any $j\in\mathbb{S}$, the vector space
\begin{equation*}
\mathbb C_j:=\{u+jv\;|\;u,v\in\mathbb{R}\}=\mathbb{R}+j\mathbb{R}
\end{equation*}
is an isomorphic copy of the complex numbers, passing through the real line in the direction of the imaginary unit $j$. Finally, a subset $U\subseteq\mathbb{R}^{n+1}$ is called \textit{axially symmetric}, if for every $x\in U$ the whole $(n-1)$-sphere $[x]\subseteq U$ is contained in $U$.

\medskip

Next we define the set of slice monogenic functions (also called slice hyperholomorphic functions), acting on paravectors $\mathbb{R}^{n+1}$ and having values in the full Clifford algebra $\mathbb{R}_n$.

\begin{definition}[Slice monogenic functions]\label{SHolDefMON}
Let $U\subseteq\mathbb{R}^{n+1}$ be an axially symmetric open set and let $\mathcal{U}=\{(u,v)\in\mathbb{R}^2: u+\mathbb{S}v\subseteq U\}$. A function $f:U\to\mathbb{R}_n$ is called \textit{left slice monogenic}, if it is of the form
\begin{equation}\label{Eq_f_decomposition}
f(x)=f_0(u,v)+jf_1(u,v),\qquad x=u+jv\in U,
\end{equation}
where the two functions $f_0,f_1:\mathcal{U}\rightarrow\mathbb{R}_n$ satisfy the compatibility conditions
\begin{equation}\label{CCondmon}
f_0(u,-v)=f_0(u,v)\qquad\text{and}\qquad f_1(u,-v)=-f_1(u,v),
\end{equation}
as well as the Cauchy-Riemann-equations
\begin{equation}\label{CRMMON}
\frac{\partial}{\partial u}f_0(u,v)=\frac{\partial}{\partial v}f_1(u,v)\qquad\text{and}\qquad\frac{\partial}{\partial v} f_0(u,v)=-\frac{\partial}{\partial u}f_1(u,v).
\end{equation}
The set of left slice monogenic functions will be denoted by $\mathcal{S\!M}_L(U)$.
\end{definition}

\begin{remark}
Analogously, one can also define right slice monogenic functions by simply replacing the decomposition \eqref{Eq_f_decomposition} by
\begin{equation*}
f(s)=f_0(u,v)+f_1(u,v)j,\qquad x=u+jv\in U.
\end{equation*}
However, for the rest of the paper we will only consider left slice monogenic functions, although all the results can also be written for right slice monogenic functions.
\end{remark}

\begin{definition}\label{slicederivative}
Let $f\in\mathcal{S\!M}_L(U)$. Then for any $x\in U$ we define the \textit{left slice derivative} as
\begin{equation}\label{SDerivL}
\partial_Sf(x):=\lim_{p\to x,p\in\mathbb{C}_j}(p-x)^{-1}(f(p)-f(x)),
\end{equation}
where for nonreal $x$ the imaginary unit $j\in\mathbb{S}$ is chosen such that $x\in\mathbb{C}_j$ and for real $x$ one may choose $j\in\mathbb{S}$ arbitrary.
\end{definition}

We remark that $\partial_Sf(x)$ is uniquely defined and independent of the choice of $j\in\mathbb{S}$, even if $x$ is real. Moreover, we note that the slice derivative of $f$ coincides with
\begin{equation}\label{SDerivPartial}
\partial_Sf(x)=\partial_{x_0}f(x)=f_j'(x),
\end{equation}
where $\partial_{x_0}f$ is the partial derivative with respect to $x_0$ and $f_j'$ is the complex derivative of the restriction $f_j=f|_{\mathbb{C}_j}$.

\begin{theorem}\label{PowSerThm}
For $a\in\mathbb{R}$, $r>0$ let $B_r(a)$ the open ball with radius $r$ centered at $a$. Then for every $f\in\mathcal{S\!M}_L(B_r(a))$ there holds
\begin{equation}\label{PowSerL}
f(x)=\sum_{k=0}^\infty(x-a)^k\frac{1}{k!}\partial_{x_0}^kf(a),\qquad x\in B_r(a).
\end{equation}
\end{theorem}

We now recall the natural product of two functions, that preserves slice monogenicity.

\begin{definition}
Let $U\subseteq\mathbb{R}^{n+1}$ be open and axially symmetric. Then for any $f,g\in\mathcal{S\!M}_L(U)$ define their \textit{star product}
\begin{equation*}
f\star_L g:=f_0g_0-f_1g_1+j(f_1g_0+f_0g_1),
\end{equation*}
with the functions $f_0,f_1,g_0,g_1$ from the decomposition \eqref{Eq_f_decomposition}.
\end{definition}

\begin{lemma}
Let $f(x)=\sum_{k=0}^\infty x^ka_k$ and $g(x)=\sum_{k=0}^\infty x^kb_k$ be two left slice monogenic power series. Then their star product is given by
\begin{equation}\label{ProdLSeries}
(f\star_L g)(x)=\sum_{\ell=0}^\infty x^\ell\sum_{k=0}^\ell a_kb_{\ell-k}.
\end{equation}
\end{lemma}

For $x,s\in\mathbb{R}^{n+1}$ with $x\not\in[s]$, the Cauchy kernel for left slice monogenic functions is given by
\begin{equation}\label{Eq_SL}
S_L^{-1}(s,x):=-(x^2-2s_0x+|s|^2)^{-1}(x-\overline{s})=(s-\overline{x})(s^2-2x_0s+|x|^2)^{-1}.
\end{equation}
With this kernel there holds the following Clifford algebra version of the Cauchy integral formula.

\begin{theorem}[The Cauchy formula]\label{CauchygeneraleMONOG}
Let $U\subset\mathbb{R}^{n+1}$ be axially symmetric, open, bounded and the boundary $\partial(U\cap\mathbb{C}_j)$ be a finite union of continuously differentiable Jordan curves. If we set $ds_j=-jds$ for some $j\in\mathbb{S}$, then for $f$ which is left slice monogenic on a neighborhood of $\overline{U}$, there holds
\begin{equation}\label{cauchynuovo}
f(x)=\frac{1}{2\pi}\int_{\partial(U\cap\mathbb{C}_j)}S_L^{-1}(s,x)ds_jf(s),\qquad x\in U.
\end{equation}
Moreover, the integral depend neither on $U$ nor on the imaginary unit $j\in\mathbb{S}$.
\end{theorem}

After the basic facts on slice monogenic functions we introduce the notion of proximate order functions which will then be used in Section \ref{Prel} to introduce function spaces of exponentially bounded entire slice monogenic functions.

\begin{definition}[Proximate order]\label{defi_Proximate_order}
A differentiable function $\varrho:[0,\infty)\to[0,\infty)$ is called a \textit{proximate order function} for the order $\rho>0$, if it satisfies

\begin{enumerate}
\item $\lim_{r\to\infty}\varrho(r)=:\rho>0$,
\item $\lim_{r\to\infty}\varrho'(r)r\ln(r)=0$.
\end{enumerate}
\end{definition}

It is also possible to define proximate order for the order $\rho=0$, but the upcoming results do not hold in this case. Any proximate order function admits the following properties, for a proof see \cite[Proposition 1.19]{LG86}.

\begin{lemma}[Properties of proximate orders]\label{lem_Properties_of_proximate_orders}
Let $\varrho$ be a proximate order. Then there exists some $r_0>0$, such that

\begin{enumerate}
\item[i)] $r\mapsto r^{\varrho(r)}$ is strictly increasing on $[r_0,\infty)$,
\item[ii)] $\lim\limits_{r\to\infty}r^{\varrho(r)}=\infty$.
\end{enumerate}
\end{lemma}

\begin{definition}[Normalized proximate order]\label{defi_Normalized_proximate_order}
A proximate order function $\varrho$ is called \textit{normalized} if

\begin{enumerate}
\item[i)] $r\mapsto r^{\varrho(r)}$ is strictly increasing on $(0,\infty)$,
\item[ii)] $\lim\limits_{r\to 0^+}r^{\varrho(r)}=0$
\end{enumerate}

For a normalized proximate order it is then clear that $r\mapsto r^{\varrho(r)}$ maps $(0,\infty)$ bijectively onto $(0,\infty)$ and we denote by
\begin{equation}\label{Eq_varphi}
\varphi:(0,\infty)\to(0,\infty)\text{ the inverse function of }r\mapsto r^{\varrho(r)}.
\end{equation}
Moreover, we set the numbers
\begin{equation}\label{Eq_Gl}
G_0=G_{\varrho,0}:=1\qquad\text{and}\qquad G_\ell=G_{\varrho,\ell}:=\frac{\varphi(\ell)^\ell}{(e\rho)^{\ell/\rho}},\qquad\ell\in\mathbb{N}.
\end{equation}
\end{definition}

\begin{remark}\label{bem_normalization}
Note that by Lemma \ref{lem_Properties_of_proximate_orders} it is possible to construct for any given proximate order function $\rho$ a normalized proximate order $\hat{\rho}$ such that $\rho(r)=\hat{\rho}(r)$ for every $r\geq r_0$ large enough. One possible choice is the function
\begin{equation*}
\hat{\varrho}(r):=\begin{cases} \varrho(r_0)-\frac{\rho}{4}\sin\big(4\frac{\varrho'(r_0)}{\rho}(r_0-r)\big), & r\in[0,r_0], \\ \varrho(r), & r\in[r_0,\infty). \end{cases}
\end{equation*}
\end{remark}

Next, we recall some basic properties of proximate orders that will be used in the following. First we prove two inequalities of the mapping $r\mapsto r^{\varrho(r)}$. The first one i) can be found in \cite[Lemma 2.3]{AIO20} and the second one ii) in \cite[Proposition 1.20]{AIO20}.

\begin{lemma}\label{l2505231}
Let $\varrho$ be a normalized proximate order function. Then for any $\varepsilon>0$ there exists a constant $C_\varepsilon\geq 0$, such that

\begin{enumerate}
\item[i)] $(r+s)^{\varrho(r+s)}\leq 2^{\rho+\varepsilon}\big(r^{\varrho(r)}+s^{\varrho(s)}\big)+C_\varepsilon,$\qquad $r,s>0$.
\item[ii)] $(sr)^{\varrho(sr)}\leq(1+\varepsilon)s^\rho r^{\varrho(r)}+C_\varepsilon$,\qquad $r,s>0$.
\end{enumerate}
\end{lemma}

The next lemma can be found in \cite[Lemma 2.4]{AIO20} and it is the submultiplicativity of the numbers $G_\ell$ in \eqref{Eq_Gl}.

\begin{lemma}\label{lsupersupernova2}
The sequence $\{G_\ell\}_{\ell\in\mathbb{N}_0}$ from \eqref{Eq_Gl} satisfies
\begin{equation*}
G_\ell G_k\leq G_{\ell+k},\qquad\ell,k\in\mathbb{N}_0.
\end{equation*}
\end{lemma}

The following Lemma treats the limit behaviour of the function $\varphi$ from Definition \ref{defi_Normalized_proximate_order}. The upcoming results i) and ii) can be found in the proof of \cite[Theorem 1.23]{LG86}, the result iii) in \cite[Proposition 1.20]{LG86} and iv) in \cite[Lemma 2.6]{AIO20}.

\begin{lemma}\label{l3005231}
Let $\rho$ be a normalized proximate order function and $\varphi$ from \eqref{Eq_varphi}. Then for every $s>0$ and $0<\sigma'<\sigma$ there holds

\begin{itemize}
\item[i)] $\lim_{t\to\infty}\frac{t\varphi'(t)}{\varphi(t)}=\frac{1}{\varrho}$,
\item[ii)] $\lim_{t\to\infty}\frac{\varphi(st)}{\varphi(t)}=s^{\frac{1}{\varrho}}$,
\item[iii)] $\lim_{r\to\infty}\frac{(sr)^{\rho(sr)}}{r^{\rho(r)}}=s^\rho$.
\item[iv)] $\exists t_0>0: \frac{\varphi(t)}{\varphi(t')}\leq\frac{e^{\sigma\frac{t}{t'}}}{(e\sigma'\varrho)^{1/\varrho}}$,\quad $t,t'\geq t_0$.
\end{itemize}
\end{lemma}

\section{Slice monogenic functions of proximate order}\label{Prel}

Let $\varrho(r)$ be a proximate order for a positive order $\rho>0$ according to Definition \ref{defi_Proximate_order} . We will now introduce some function spaces of entire slice monogenic functions in the spirit of \cite{LG86,AIO20}, which treat the similar problem for entire functions of several complex variables. For constant proximate order functions $\varrho(r)=\rho$ and slice monogenic functions in the quaternions we also refer to the results in \cite[Chapter 5]{ACSBOOK2}.  In particular we derive basic properties of these function spaces where most of them are already known in the complex and the monogenic setting, but which are new in the case of slice monogenic functions. Moreover, we give an alternative and more detailed proof of \cite[Theorem 1.23]{LG86} in Theorem \ref{thm_tsupernovissima1}.

\medskip

For any $\sigma>0$, we consider the Banach space
\begin{equation*}
A_{\varrho,\sigma}:=\Big\{f\in\mathcal{S\!M}_L(\mathbb{R}^{n+1})\;|\;\Vert f\Vert_{\varrho,\sigma}:=\sup_{x\in\mathbb{R}^{n+1}}|f(x)|\exp(-\sigma|x|^{\varrho(|x|)})<\infty\Big\}
\end{equation*}
with the norm $\Vert\cdot\Vert_{\varrho,\sigma}$.

\begin{remark}\label{equivalencenormalization}
We observe that for any proximate function $\varrho(r)$ and any normalization function $\hat\varrho(r)$ according to Remark \ref{bem_normalization}, the spaces $A_{\varrho,\sigma}$ and $A_{\hat \varrho,\sigma}$ coincide with equivalent norms, see \cite[Page 8]{AIO20}.
\end{remark}

\begin{lemma}\label{lsupernova1}
If $\sigma_2>\sigma_1>0$, then the inclusion map $A_{\varrho,\sigma_1}\xhookrightarrow{} A_{\varrho,\sigma_2}$ is compact.
\end{lemma}

\begin{proof}
We will show that $B:=\{f\in A_{\varrho,\sigma_1}\;|\;\Vert f\Vert_{\varrho,\sigma_1}\leq 1\}$ is relatively compact in $A_{\varrho,\sigma_2}$, i.e., any sequence $\{f_k\}_{k\in\mathbb N}\subset B$ has an accumulation point with respect to the norm of $A_{\varrho,\sigma_2}$.

First we will prove that the sequence $\{f_k\}_{k\in\mathbb N}\subset B$ admits a subsequence which converges in the uniform convergence topology of $\mathbb{R}^{n+1}$. By the Arzel\'a-Ascoli theorem, together with some standard diagonal sequence argument, it is sufficient to prove that $\{f_k\}_{k\in\mathbb N}$ is equicontinuous and uniformly bounded on any compact convex subset $K\subseteq\mathbb R^{n+1}$. Let us now fix one of the compact convex subsets $K$. Since $\{f_k\}_{k\in\mathbb N}\subset B$, the sequence is uniformly bounded on $K$. Moreover, we have
\begin{equation*}
|f_k(x)-f_k(y)|\leq C_k|x-y|,\qquad x,y\in K
\end{equation*}
where $C_k=\sup_{x\in K}|\nabla f_k(x)|$ and $\nabla$ is the usual gradient. We choose $r$ large enough in a such way that $K\subset B(0,r)$. It is sufficient to prove that there exists a constant $C_K$, only depending on $K$, such that
\begin{equation*}
|\partial_{x_i}f_k(x)|\leq C_K,\qquad x\in K,\,k\in\mathbb{N},\,i=0,\dots,n.
\end{equation*}
To prove this fact we need to differentiate the integral representation formula \eqref{cauchynuovo} for $f_k$. Let $x\in K$ and $j\in \mathbb S$ be such that there exist $u,\, v\in\mathbb R$ with $x=u+jv$. When $i=0$, there exists a positive constant $C'_K$ which depends only on $K$, such that
\begin{align*}
|\partial_{x_0}f_k(x)|&=\frac{1}{2 \pi}\Big|\int_{\partial (B(0,r)\cap\mathbb{C}_j)}\partial_{x_0}S_L^{-1}(s,x)ds_jf_k(s)\Big| \\
&=\frac{1}{2\pi}\Big|\int_{\partial (B(0,r)\cap\mathbb{C}_j)}\frac{1}{(s-x)^2}ds_jf_k(s)\Big| \\
&\leq\frac{M(r,f_k)}{2\pi}\int_{\partial(B(0,r)\cap\mathbb C_j)}\frac{1}{|x-s|^2}|ds_j| \\
&\leq\frac{rM(r,f_k)}{(r-|x|)^{2}}\leq C'_K,
\end{align*}
where from $\{f_k\}_{k\in\mathbb N}\subset B$ and $\operatorname{dist}(K,\partial B(0,r))>0$ the last inequality follows. When $i\neq 0$, using the second form for $S^{-1}_L(s,x)$ in \eqref{Eq_SL}, we observe that
\begin{equation*}
\partial_{x_i}S_L^{-1}(s,x)=(e_is^2-2x_0e_is+|x|^2e_i-2x_is+2x_i\overline{x})(s^2-2x_0s+|x|^2)^{-2}.
\end{equation*}
Thus, when $x\in\mathbb{C}_j\cap K$ and $s\in\mathbb{C}_j\cap\partial B(0,r)$, there exists a positive constant $C_2$, which depends only on $K$, such that
\begin{equation*}
|\partial_{x_i}S_L^{-1}(s,x)|\leq\frac{C_2}{|s-x|^2|s-\overline{x}|^2}.
\end{equation*}
By this inequality, there exists another positive constant $C''_K$, which only depends on $K$, such that
\begin{align*}
|\partial_{x_i}f_k(x)|&=\frac{1}{2\pi}\Big|\int_{\partial(B(0,r)\cap\mathbb{C}_j)}\partial_{x_i}S_L^{-1}(s,x)ds_jf_k(s)\Big| \\
&\leq\frac{M(r,f_k)}{2\pi}\int_{\partial(B(0,r)\cap\mathbb{C}_j)}|\partial_{x_i}S_L^{-1}(s,x)||ds_j| \\
&\leq rC_2\frac{M(r,f_k)}{(r-|x|)^{4}}\leq C''_K.
\end{align*}
In particular, choosing $C_K=\max\{C'_K,C''_K\}$, for any $x,y\in K$ and for any $k\in\mathbb N$ we have
\begin{equation*}
|f_k(x)-f_k(y)|\leq\sqrt{n}C_K|x-y|,
\end{equation*}
i.e., $\{f_k\}_{k\in\mathbb N}$ is equicontinuous on $K$. Applying the Arzel\'a-Ascoli theorem and using some standard diagonal sequence argument, there exists a subsequence which converges to $f\in\mathcal{S\!M}_L(\mathbb R^{n+1})$ in the topology of the uniform convergence on compact subsets. Without loss of generality we call this subsequence again by $\{f_k\}_{k\in\mathbb{N}}$.

\medskip

In the second step we prove that the sequence $\{f_k\}_{k\in\mathbb N}$ is a Cauchy sequence in $A_{\varrho,\sigma_2}$. We fix $\delta>0$ and choose $R>0$ large enough such that
\begin{equation*}
\exp\big((\sigma_1-\sigma_2)|x|^{\varrho(|x|)}\big)\leq\frac{\delta}{2},\qquad\text{for any }|x|\geq R.
\end{equation*}
Thus, since $f_k,f_\ell\in B$, we have
\begin{align}
\sup_{|x|\geq R}&|f_k(x)-f_\ell(x)|\exp\big(-\sigma_2|x|^{\varrho(|x|)}\big) \notag \\
&=\sup_{|x|\geq R}|f_k(x)-f_\ell(x)|\exp\big(-\sigma_1|x|^{\varrho(|x|)}\big)\exp\big((\sigma_1-\sigma_2)|x|^{\varrho(|x|)}\big) \notag \\
&\leq 2\,\frac{\delta}{2}=\delta. \label{Eq_Compact_inclusion_1}
\end{align}
Moreover, by the uniform convergence of the sequence $\{f_k\}_{k\in\mathbb N}$ on the compact subset $\overline{B(0,R)}$ of $\mathbb{R}^{n+1}$, there exists a positive integer $N$ such that for any $k,\ell\geq N$ we have
\begin{equation}\label{Eq_Compact_inclusion_2}
\sup_{|x|\leq R}|f_k(x)-f_\ell(x)|\exp\big(-\sigma_2|x|^{\varrho(|x|)}\big)\leq\sup_{|x|\leq R}|f_k(x)-f_\ell(x)|\leq\delta.
\end{equation}
Thus, combining \eqref{Eq_Compact_inclusion_1} and \eqref{Eq_Compact_inclusion_2} we have proved that the sequence $\{f_k\}_{k\in\mathbb N}$ is a Cauchy sequence in $A_{\varrho,\sigma_2}$.
\end{proof}

\begin{definition}[The spaces $A_\varrho$ and $A_{\varrho,\sigma+0}$]\label{defi_Arho0}
We define the space
\begin{equation*}
A_\varrho:=\lim_{\underset{\sigma>0}{\to}}A_{\varrho,\sigma},
\end{equation*}
i.e. $A_\varrho=\bigcup_{\sigma>0}A_{\varrho,\sigma}$ and we say that a sequence $\{f_k\}_{k\in\mathbb{N}}\subseteq A_\varrho$ converges to $f\in A_\varrho$ if there exists $\sigma>0$ such that $\{f_k\}_{k\in\mathbb{N}}\subseteq A_{\varrho,\sigma}$, $f\in A_{\varrho,\sigma}$ and $\lim_{k\to\infty}\Vert f_k-f\Vert_{\varrho,\sigma}=0$.

For every $\sigma\geq 0$ we also define the space
\begin{equation*}
A_{\varrho,\sigma+0}:=\lim_{\underset{\epsilon>0}{\leftarrow}}A_{\varrho,\sigma+\epsilon},
\end{equation*}
i.e., $A_{\varrho,\sigma+0}:=\bigcap_{\epsilon>0} A_{\varrho,\sigma+\epsilon}$ and we say that a sequence $\{f_k\}_{k\in\mathbb{N}}\subseteq A_{\varrho,\sigma+0}$ converges to $f\in A_{\varrho,\sigma+0}$ if for any $\epsilon>0$ we have $\{f_k\}_{k\in\mathbb N}\subseteq A_{\varrho,\sigma+\epsilon}$, $f\in A_{\varrho,\sigma+\epsilon}$ and $\lim_{k\to\infty}\Vert f_k-f\Vert_{\varrho,\sigma+\epsilon}=0$. When $\sigma=0$, we denote $A_{\varrho,+0}:=A_{\varrho,0+0}$.
\end{definition}

Note that $A_\varrho$ is a (DFS)-space (see \cite[Definition 2.2.1 and Section 2.6]{BG95}) and $A_{\varrho,\sigma+0}$ is a (FS)-space. Both $A_\varrho$ and $A_{\varrho,\sigma+0}$ are locally convex spaces.

\begin{remark}
Let $\varrho$ be a proximate order function and $\hat{\rho}$ its normalization according to Remark \ref{bem_normalization}. Then the spaces $A_\varrho$ and $A_{\hat{\varrho}}$ coincide and share the same locally convex topology. The same holds true for $A_{\varrho,\sigma+0}$ and $A_{\hat{\varrho},\sigma+0}$.
\end{remark}

\begin{definition}
Let $\varrho$ be a proximate order function and $f\in A_\varrho$. Then we define the \textit{type} of $f$ with respect to $\varrho$ as
\begin{equation*}
\inf\{\beta>0\;|\;f\in A_{\varrho,\beta}\}.
\end{equation*}
\end{definition}

Next we prove a preparatory lemma for Theorem \ref{thm_tsupernovissima1}, which characterises the order.

\begin{lemma}\label{l0706231}
Let $\varrho$ be a normalized proximate order for the positive order $\rho$. Let $f(x)=\sum_{\ell=0}^\infty x^\ell a_\ell\in A_\varrho$. Moreover, suppose that $\sigma\geq 0$ which satisfies the property
\begin{equation*}
\frac{1}{\rho}\ln(\sigma)\geq\limsup_{\ell\to\infty}\Big(\frac{1}{\ell}\ln|a_\ell|+\ln(\varphi(\ell))\Big)-\frac{1}{\rho}-\frac{\ln(\rho)}{\rho},
\end{equation*}
where we interpret $\ln(0)=-\infty$ in the case $\sigma=0$. Then, for any $\tau'>\sigma$, there exist positive constants $N$ and $C$ such that
\begin{equation}\label{Eq_Arhosigma_characterisation_6}
\sup\limits_{\ell\geq N}\big(\ln|a_\ell|+\ell\ln(r)\big)\leq \tau'r^{\rho(r)}+C,\qquad r>0.
\end{equation}
\end{lemma}

\begin{proof}
Choose $\sigma<\sigma'<\sigma''<\tau<\tau'$ arbitrary. Then, by assumption there exists $N_1\in\mathbb{N}$, such that
\begin{equation*}
|a_\ell|^{\frac{1}{\ell}}\varphi(\ell)\leq(e\rho\sigma')^{\frac{1}{\rho}},\qquad\ell\geq N_1.
\end{equation*}
Moreover, by the limit $\lim_{\ell\to\infty}\frac{\varphi(\frac{\ell}{\rho\sigma''})}{\varphi(\ell)}=(\frac{1}{\rho\sigma''})^{\frac{1}{\rho}}$ from Lemma \ref{l3005231} (ii), there exists some $N_2\geq N_1$ such that $\frac{\varphi(\frac{\ell}{\sigma''\rho})}{\varphi(\ell)}\leq(\frac{1}{\rho\sigma'})^{\frac{1}{\rho}}$ and hence
\begin{equation}\label{Eq_Arhosigma_characterisation_3}
|a_\ell|^{\frac{1}{\ell}}\leq\frac{(e\rho\sigma')^{\frac{1}{\rho}}}{\varphi(\ell)}\leq\frac{e^{\frac{1}{\rho}}}{\varphi(\frac{\ell}{\sigma''\rho})},\qquad\ell\geq N_2.
\end{equation}
Next, by Lemma \ref{l3005231} (i), there exists some $N\geq N_2$ such that
\begin{equation}\label{Eq_Arhosigma_characterisation_1}
-1\leq\frac{\frac{t}{\sigma''\rho}\varphi'(\frac{t}{\sigma''\rho})}{\varphi(\frac{t}{\sigma''\rho})}-\frac{1}{\rho}\leq\frac{\tau-\sigma''}{\rho\sigma''},\qquad t\geq N.
\end{equation}
We will now prove that there exists a positive constant $C$, which does not depend on $r$, and $r_0>0$, such that
\begin{equation}\label{Eq_Arhosigma_characterisation_4}
\sup\limits_{\ell\geq N}\big(\ln|a_\ell|+\ell\ln(r)\big)\leq\tau r^{\rho(r)} + C,\quad r>r_0.
\end{equation}
Due to the estimate \eqref{Eq_Arhosigma_characterisation_3} we get
\begin{align}
\sup\limits_{\ell\geq N}\big(\ln|a_\ell|+\ell\ln(r)\big)&\leq\sup\limits_{\mathbb{N}\ni\ell\geq N}\ell\Big(\frac{1}{\rho}-\ln\Big(\varphi\Big(\frac{\ell}{\sigma''\rho}\Big)\Big)+\ln(r)\Big) \notag \\
&\leq\sup\limits_{\mathbb{R}\ni t\geq N}\underbrace{t\Big(\frac{1}{\rho}-\ln\Big(\varphi\Big(\frac{t}{\sigma''\rho}\Big)\Big)+\ln(r)\Big)}_{\eqqcolon\mu_r(t)}. \label{Eq_Arhosigma_characterisation_2}
\end{align}
Let $t_\text{max}(r)$ be the supremum of the points where the function $\mu_r$ attains its maximum. We observe that this point exists and it is finite since $\mu_r$ is continuous and converges to $-\infty$ when $t\to\infty$. First we prove that $t_\text{max}(r)\to\infty$ when $r\to\infty$. We observe that
\begin{equation*}
\mu_r'(t)=\frac{1}{\rho}-\ln\Big(\varphi\Big(\frac{t}{\sigma''\rho}\Big)\Big)+\ln(r)-\frac{\frac{t}{\sigma''\rho}\varphi'(\frac{t}{\sigma''\rho})}{\varphi(\frac{t}{\sigma''\rho})}.
\end{equation*}
We assume by contradiction that $t_\text{max}(r)$ is bounded. Since $\mu'_r(t_\text{max}(r))\leq 0$, we have
\begin{equation}\label{Eq_Arhosigma_characterisation_7}
\ln\bigg(\varphi\Big(\frac{t_\text{max}(r)}{\sigma''\rho}\Big)\bigg)\geq\frac{1}{\rho}+\ln(r)-\frac{\frac{t_\text{max}(r)}{\sigma''\rho}\varphi'(\frac{t_\text{max}(r)}{\sigma''\rho})}{\varphi(\frac{t_\text{max}(r)}{\sigma''\rho})}.
\end{equation}
The previous inequality gives a contradiction since the left hand side is bounded for any $r>0$, instead the right hand side tends to $\infty$ when $r\to\infty$.

\medskip

Since we just have proven that $t_\text{max}(r)\overset{r\to\infty}{\longrightarrow}\infty$ there exists some $r_0$ such that $t_\text{max}(r)>N$ for any $r\geq r_0$ and hence also $\mu_r'(t_\text{max}(r))=0$ has to be satisfied for any $r\geq r_0$. This means that the inequality \eqref{Eq_Arhosigma_characterisation_7} beomes an equation, i.e.
\begin{equation*}
\ln\bigg(\frac{1}{r}\varphi\Big(\frac{t_\text{max}(r)}{\sigma''\rho}\Big)\bigg)=\frac{1}{\rho}-\frac{\frac{t_\text{max}(r)}{\sigma''\rho}\varphi'(\frac{t_\text{max}(r)}{\sigma''\rho})}{\varphi(\frac{t_\text{max}(r)}{\sigma''\rho})}.
\end{equation*}
In view of Lemma \ref{l3005231} (i), we have
\begin{equation*}
\lim_{r\to\infty}\ln\bigg(\frac{1}{r}\varphi\Big(\frac{t_\text{max}(r)}{\sigma''\rho}\Big)\bigg)=0.
\end{equation*}
Next we choose $\epsilon,\eta>0$ such that $(1+\eta)(1+\epsilon)^\rho\tau\leq\tau'$. By the previous limit, we can enlarge $r_0>0$ such that
\begin{equation*}
\varphi\Big(\frac{t_\text{max}(r)}{\sigma''\rho}\Big)\leq(1+\epsilon)r,\qquad r\geq r_0.
\end{equation*}
By applying the inverse function $\varphi^{-1}(r)=r^{\varrho(r)}$ to both sides of the this inequality and by Lemma \ref{l2505231} ii) with the above chosen $\eta$, we get
\begin{equation}\label{e0606231}
t_\text{max}(r)\leq\sigma''\rho((1+\epsilon)r)^{\varrho((1+\epsilon)r)}\leq\sigma''\rho(1+\eta)(1+\epsilon)^\rho r^{\varrho(r)}+\sigma''\rho C_\eta.
\end{equation}
Moreover, rearranging the equation $\mu_r'(t_\text{max}(r))=0$ and using the second inequality in \eqref{Eq_Arhosigma_characterisation_1}, we have that
\begin{equation}\label{e09061753}
\ln\bigg(\varphi\Big(\frac{t_\text{max}(r)}{\sigma''\rho}\Big)\bigg)=\ln(r)+\frac{1}{\rho}-\frac{\frac{t_\text{max}(r)}{\sigma''\rho}\varphi'(\frac{t_\text{max}(r)}{\sigma''\rho})}{\varphi(\frac{t_\text{max}(r)}{\sigma''\rho})}\geq\ln(r)-\frac{\tau-\sigma''}{\rho\sigma''}.
\end{equation}
Using the inequalities \eqref{e0606231} and \eqref{e09061753} in \eqref{Eq_Arhosigma_characterisation_2}, we obtain for every $r\geq r_0$
\begin{align}
\sup\limits_{\ell\geq N}\big(&\ln|a_\ell|+\ell\ln(r)\big)=\mu_r(t_\text{max}(r)) \notag \\
&=t_\text{max}(r)\Big(\frac{1}{\rho}-\ln\Big(\varphi\Big(\frac{t_\text{max}(r)}{\sigma''\rho}\Big)\Big)+\ln(r)\Big) \notag \\
&\leq(\sigma''\rho(1+\eta)(1+\epsilon)^\rho r^{\varrho(r)}+\sigma''\rho C_\eta)\Big(\frac{1}{\rho}-\ln(r)+\frac{\tau-\sigma''}{\rho\sigma''}+\ln(r)\Big) \notag \\
&\leq\tau'r^{\rho(r)}+\tau'C_\eta. \label{e09061816}
\end{align}
Since $\ln(r)$ is increasing we furthermore get the estimate
\begin{equation*}
\sup\limits_{\ell\geq N}\big(\ln|a_\ell|+\ell\ln(r)\big)\leq\sup\limits_{\ell\geq N}\big(\ln|a_\ell|+\ell\ln(r_0)\big)\leq\tau'r_0^{\rho(r_0)}+\tau'C_\eta,\quad r\in(0,r_0].
\end{equation*}
Hence, choosing $C'=\tau'r_0^{\rho(r_0)}+\tau'C_\eta$, we finally have
\begin{equation*}
\sup\limits_{\ell\geq N}\big(\ln|a_\ell|+\ell\ln(r)\big)\leq \tau'r^{\rho(r)}+C',\qquad r>0. \qedhere
\end{equation*}
\end{proof}

Next we prove the main theorem of this section, a characterization of functions in the spaces $A_{\varrho,\sigma+0}$ with respect to the order, their growth condition and their Taylor coefficients.

\begin{theorem}\label{thm_tsupernovissima1}
Let $\varrho$ be a normalized proximate order function, $\sigma\geq 0$ and $f(x)=\sum_{\ell=0}^\infty x^\ell a_\ell\in\mathcal{S\!M}_L(\mathbb{R}^{n+1})$. Then the following four statements are equivalent:

\begin{enumerate}
\item $f\in A_{\varrho,\sigma+0}$; \\
\item $\limsup_{r\to\infty}\frac{\sup_{|x|\leq r}\ln|f(x)|}{r^{\varrho(r)}}\leq\sigma$; \\
\item $\limsup_{\ell\to\infty}|a_\ell|^{\frac{1}{\ell}}\varphi(\ell)\leq(e\rho\sigma)^{\frac{1}{\rho}}$; \\
\item $\inf\{\beta>0\;|\;f\in A_{\varrho,\beta}\}\leq\sigma$.
\end{enumerate}
\end{theorem}

\begin{remark}\label{r1006231857}
Note that using the numbers $\{G_\ell\}_{\ell\in\mathbb{N}_0}$ from \eqref{Eq_Gl}, one can alternatively write \textit{(3)} as
\begin{equation*}
\limsup_{\ell\to\infty}(|a_\ell|G_{\varrho,\ell})^{\frac{\rho}{\ell}}\leq\sigma.
\end{equation*}
\end{remark}

\begin{proof}[Proof of Theorem \ref{thm_tsupernovissima1}]
In the first step we prove the equivalence $(2)\Leftrightarrow(4)$. First, we prove
\begin{equation}\label{e09061925}
\limsup_{r\to\infty}\frac{\sup_{|x|\leq r}\ln|f(x)|}{r^{\varrho(r)}}\geq\inf\{\beta>0\;|\;f\in A_{\varrho,\beta}\}.
\end{equation}
If $\tau>\limsup_{r\to\infty}\frac{\sup_{|x|\leq r}\ln|f(x)|}{r^{\varrho(r)}}$, then we can choose $r_0>0$ such that for any $|x|>r_0$ we have
\begin{equation*}
|f(x)|\leq\exp\big(\tau|x|^{\varrho(|x|)}\big).
\end{equation*}
Thus there exists a positive constant $C$ such that $|f(x)|\leq C\exp(\tau|x|^{\varrho(|x|)})$ for any $x\in\mathbb{R}^{n+1}$ and $f\in A_{\varrho,\tau}$ i.e $\tau\geq\inf\{\beta>0\;|\;f\in A_{\varrho,\beta}\}$. Hence \eqref{e09061925} is proven. Now we prove that
\begin{equation*}
\limsup_{r\to\infty}\frac{\sup_{|x|\leq r}\ln|f(x)|}{r^{\varrho(r)}}\leq\inf\{\beta>0\;|\;f\in A_{\varrho,\beta}\}.
\end{equation*}
If $\sigma<\limsup_{r\to\infty}\frac{\sup_{|x|\leq r}\ln|f(x)|}{r^{\varrho(r)}}$ then, given an increasing divergent sequence of positive constant $\{C_k\}_{k\in\mathbb N}$, there exists a sequence $\{x_k\}_{k\in\mathbb N}\subset \mathbb{R}^{n+1}$ such that
\begin{equation*}
|f(x_k)|\geq C_k\exp\big(\sigma|x_k|^{\varrho(|x_k|)}\big).
\end{equation*}
This can be proved observing that for any $C_k>0$ we have
\begin{equation*}
\sigma<\limsup_{r\to\infty}\frac{\sup_{|x|\leq r}\ln|f(x)|}{r^{\varrho(r)}}=\limsup_{r\to\infty}\frac{\sup_{|x|\leq r}\ln|f(x)|}{r^{\varrho(r)}}-\frac{\ln(C_k)}{r^{\varrho(r)}}.
\end{equation*}
Thus $f\notin A_{\varrho,\sigma}$ i.e $\sigma\leq\inf\{\beta>0\;|\;f\in A_{\varrho,\beta}\}$.

\medskip

Next we prove the equivalence $(1)\Leftrightarrow(2)$. We have that $f\in A_{\varrho,\sigma+0}$ if and only if for any $\epsilon>0$ there exists a $D_\epsilon>0$ such that
\begin{equation}\label{newnova1bis}
|f(x)|\leq D_\epsilon\exp\big((\sigma+\epsilon)|x|^{\varrho(|x|)}\big),\qquad x\in\mathbb{R}^{n+1}.
\end{equation}
This is equivalent to
\begin{equation}\label{newnova1}
\limsup_{r\to\infty}\frac{\sup_{|x|\leq r}\ln|f(x)|}{r^{\varrho(r)}}\leq\sigma.
\end{equation}
Indeed, supposing \eqref{newnova1bis} is true. For a fixed $\epsilon>0$ there exists $D_\epsilon>0$ such that
\begin{equation*}
\ln|f(x)|\leq\ln(D_\epsilon)+(\sigma+\epsilon)|x|^{\varrho(|x|)}.
\end{equation*}
This implies that
\begin{equation*}
\sup_{|x|<r}\ln|f(x)|\leq\ln(D_\epsilon)+(\sigma+\epsilon)r^{\varrho(r)}.
\end{equation*}
Since $r^{\varrho(r)}\to\infty$ as $r\to\infty$, we have
\begin{equation}\label{newnova1tris}
\limsup_{r\to\infty}\frac{\sup_{|x|\leq r}\ln|f(x)|}{r^{\varrho(r)}}\leq\sigma+\epsilon.
\end{equation}
Because inequality \eqref{newnova1tris} holds to be true for any $\epsilon>0$, \eqref{newnova1} is satisfied. Vice versa, if inequality \eqref{newnova1} is supposed to be true, then for any $\epsilon >0$ there exists $r_0>0$ such that for any $r\geq r_0$ we have
\begin{equation*}
\frac{\sup_{|x|\leq r}\ln|f(x)|}{r^{\varrho(r)}}\leq\sigma+\epsilon.
\end{equation*}
Thus, for any $|x|\geq r_0$, we have
\begin{equation*}
|f(x)|\leq\exp\big((\sigma+\epsilon)|x|^{\varrho(|x|)}\big).
\end{equation*}
Choosing $D_\epsilon>1$ in such a way that $|f(x)|\leq D_\epsilon$ for any $|x|\leq r_0$, we have that \eqref{newnova1bis} is true for any $x\in\mathbb{R}^{n+1}$. This verifies the equivalence of \eqref{newnova1bis} and \eqref{newnova1} and hence also $(1)\Leftrightarrow(2)$ is proven.

\medskip

Now we prove $(1)\Rightarrow(3)$. Let $f\in A_{\rho,\sigma+0}$. First we verify that
\begin{equation*}
\frac{1}{\rho}\ln(\sigma)\geq\limsup_{\ell\to\infty}\Big(\frac{1}{q}\ln|a_\ell|+\ln(\varphi(\ell))\Big)-\frac{1}{\rho}-\frac{\ln(\rho)}{\rho},
\end{equation*}
where we use the convention $\ln(0)=-\infty$ in the case $\sigma=0$. If follows from estimating the Cauchy integral formula
\begin{equation*}
a_\ell=\frac{1}{2\pi}\int_{\partial(U_r(0)\cap\mathbb{C}_j)}s^{-\ell-1}ds_jf(s),
\end{equation*}
that
\begin{equation*}
|a_\ell|\leq\frac{M(r,f)}{r^\ell},\qquad\ell\in\mathbb{N}_0.
\end{equation*}
If $\tilde{\sigma}>\sigma$ then for $r$ large we have
\begin{equation*}
M(r,f)\leq\exp\big(\tilde{\sigma}r^{\varrho(r)}\big),
\end{equation*}
and
\begin{equation}\label{kq11}
\ln|a_\ell|\leq\tilde{\sigma}r^{\varrho(r)}-\ell\ln(r).
\end{equation}
If $\ell$ is large enough, then we define $r_\ell$ to be the real number such that $\ell=\tilde{\sigma}\rho r_\ell^{\varrho(r_\ell)}$ and we have $\varphi(\frac{\ell}{\tilde{\sigma}\rho})=r_\ell$. Thus, for $\ell$ large enough we use \eqref{kq11} for $r=r_\ell$ and divide by $\ell$ and sum $\ln(\varphi(\ell))$ to both sides of inequality. Thus we have
\begin{equation}\label{kq22}
\ln\big(\varphi(\ell)|a_\ell|^{\frac{1}{\ell}}\big)<\frac{1}{\rho}+\ln\bigg(\frac{\varphi(\ell)}{\varphi(\frac \ell{\tilde{\sigma}\rho})}\bigg).
\end{equation}
By Lemma \ref{l3005231} (ii), we have
\begin{equation*}
\lim_{\ell\to\infty}\frac{\varphi(\ell)}{\varphi(\frac{\ell}{\tilde{\sigma}\rho})}=(\tilde{\sigma}\rho)^{\frac{1}{\rho}}.
\end{equation*}
Moreover, taking the $\limsup\limits_{l\to\infty}$ to both side of \eqref{kq22}, we have
\begin{equation*}
\limsup_{\ell\to\infty}\ln\big(\varphi(\ell)|a_\ell|^{\frac{1}{\ell}}\big)\leq\ln\big((e\tilde{\sigma}\rho)^{\frac{1}{\rho}}\big)
\end{equation*}
for any $\tilde\sigma>\sigma$. Thus we have proven that
\begin{equation*}
\limsup_{\ell\to\infty}\ln\big(\varphi(\ell)|a_\ell|^{\frac{1}{\ell}}\big)\leq\ln\big((e\sigma\rho)^{\frac{1}{\rho}}\big).
\end{equation*}
For the implication $(3)\Rightarrow (1)$ let $\tau>\sigma$ and choose $\tau>\tau'>\sigma$. We want to show that $f\in A_{\varrho,\tau}$ by estimating in a suitable way $\sum_{\ell=0}^\infty|a_\ell||x|^\ell$. In what follows we are going to split the previous summation in three parts. We know that
\begin{equation*}
\limsup_{\ell\to\infty}\ln\big(\varphi(\ell)|a_\ell|^{\frac{1}{\ell}}\big)\leq\ln\big((e\sigma\rho)^{\frac{1}{\rho}}\big).
\end{equation*}
Thus by Lemma \ref{l0706231} there exist positive constants $N$ and $C$ such that
\begin{equation*}
\sup\limits_{\ell\geq N}\big(\ln|a_\ell|+\ell\ln(r)\big)\leq \tau'r^{\varrho(r)}+C,\qquad r>0.
\end{equation*}
Moreover, let $\rho_1>\rho$ and fix $x\in\mathbb{R}^{n+1}$ with $r:=|x|\geq r_0$ for some $r_0>0$ large enough. Then define $m_r:=\lfloor 2e\tau\rho_1r^{\rho_1}\rfloor$. As it is shown in \eqref{Eq_Arhosigma_characterisation_3} in such a way that for any $\ell\geq m_r$ we have
\begin{equation}\label{Eq_Arhosigma_characterisation_5}
|a_\ell|^{\frac{1}{\ell}}\leq\frac{(e\rho_1\tau)^{\frac{1}{\rho_1}}}{\varphi(\ell)}\leq\frac{2^{-1}}{\varphi(\frac{\ell}{2e\tau\rho_1})}\leq\frac{2^{-1}}{\varphi(r^{\rho_1})}\leq\frac{2^{-1}}{\varphi(r^{\varrho(r)})}=\frac{2^{-1}}r,\qquad l\geq m_r.
\end{equation}
Hence we have
\begin{equation*}
|f(x)|\leq\sum\limits_{\ell=0}^{N-1}|a_\ell||x|^\ell+\sum\limits_{\ell=N}^{m_r}|a_\ell||x|^\ell+\sum\limits_{\ell=m_r+1}^\infty|a_n||x|^\ell.
\end{equation*}
We want to estimate all the three terms of the previous summation. Since the number of terms in the first summation is finite and it does not depend on $|x|$, there exists a positive constant $C''$ that depends only on $N$ such that for any $r>0$ we have
\begin{equation*}
\sum\limits_{\ell=0}^{N-1}|a_\ell||x|^\ell\leq C''\exp\big(\tau'|x|^{\varrho(|x|)}\big).
\end{equation*}
For the second summation, using \eqref{Eq_Arhosigma_characterisation_6}, we get
\begin{equation*}
\sum\limits_{\ell=N}^{m_r}|a_\ell||x|^\ell\leq m_re^C\exp\big(\tau'|x|^{\varrho(|x|)}\big)
\end{equation*}
For the third equation we use \eqref{Eq_Arhosigma_characterisation_5} to obtain the estimate
\begin{equation*}
\sum\limits_{\ell=m_r+1}^\infty|a_\ell||x|^\ell\leq \sum_{\ell=0}^\infty 2^{-\ell}=2.
\end{equation*}
Summing up the three previous inequalities, there exists a positive constant $C$ such that
\begin{align*}
|f(x)|&\leq C''\exp\big(\tau'|x|^{\varrho(|x|)}\big)+m_rC'''\exp\big(\tau'|x|^{\varrho(|x|)}\big)+2 \\
&\leq(C''+CC'''+2)\exp\big(\tau|x|^{\varrho(|x|)}\big),
\end{align*}
where in the last inequality we chose $C>0$ independent of $r$ large enough such that
\begin{equation*}
m_r\exp\big(-(\tau-\tau')r^{\rho(r)}\big)\leq C,
\end{equation*}
which is possible since $m_r=\lfloor 2e\tau\rho_1r^{\rho_1}\rfloor$ grows polynomially in $r$. Since this is true for every $|x|\geq r_0$, this implies $f\in A_{\varrho,\tau}$.
\end{proof}

The following proposition is a direct consequence of Theorem \ref{thm_tsupernovissima1}.

\begin{proposition}\label{prop_tsupernovissima1}
Let $\varrho$ be a normalized proximate order function and consider $f(x)=\sum_{\ell=0}^\infty x^\ell a_\ell\in\mathcal{S\!M}_L(\mathbb{R}^{n+1})$. Then the following four statements are equivalent:

\begin{enumerate}
\item $f\in A_\varrho$; \\
\item $\limsup_{r\to\infty}\frac{\sup_{|x|\leq r}\ln|f(x)|}{r^{\varrho(r)}}<\infty$; \\
\item $\limsup_{\ell\to\infty}|a_\ell|^{\frac{1}{\ell}}\varphi(\ell)<\infty$; \\
\item $\inf\{\beta>0\;|\;f\in A_{\varrho,\beta}\}<\infty$.
\end{enumerate}
\end{proposition}

In Section \ref{SEC3} we will need some estimates on the norms of monomials, which will be provided in the following lemma.

\begin{lemma}\label{lsupersupernova1}
Let $\varrho$ be a normalized proximate order function. Then for every $0<\sigma'<\sigma$, there exists a constant $C(\sigma,\sigma')$ such that
\begin{equation}\label{Eq_Monomial_norm}
\Vert x^\ell\Vert_{\varrho,\sigma}\leq C(\sigma,\sigma')\frac{G_\ell}{\sigma'^{\ell/\varrho}},\qquad\ell\in\mathbb{N}_0.
\end{equation}
\end{lemma}

\begin{proof}
First of all, by Lemma \ref{l3005231} there exists some $t_0\geq 0$, such that
\begin{equation}\label{Eq_Monomial_norm_1}
\frac{\varphi(t)}{\varphi(t')}\leq\frac{e^{\sigma\frac{t}{t'}}}{(e\sigma'\varrho)^{1/\varrho}},\qquad t,t'\geq t_0.
\end{equation}
Let now $r_0:=\varphi(t_0)$. Since $\lim_{t\to\infty}\varphi(t)=\infty$, there exists another constant $t_1\geq 0$ such that
\begin{equation}\label{Eq_Monomial_norm_2}
\varphi(\ell)\geq r_0(e\rho\sigma')^{\frac{1}{\rho}},\qquad\ell\geq t_1.
\end{equation}
Let now $\ell\geq t_1$. Then for every $|x|\geq r_0$ we set $t:=|x|^{\rho(|x|)}$ and get
\begin{equation*}
\frac{|x|^\ell}{G_\ell}\exp\big(-\sigma|x|^{\varrho(|x|)}\big)=\frac{|x|^\ell(e\rho)^{\frac{\ell}{\rho}}}{\varphi(\ell)^\ell}\exp\big(-\sigma|x|^{\varrho(|x|)}\big)=\frac{\varphi(t)^\ell(e\rho)^{\frac{\ell}{\rho}}}{\varphi(\ell)^\ell}e^{-\sigma t}.
\end{equation*}
Since $t=|x|^{\rho(|x|)}\geq r_0^{\rho(r_0)}=t_0$, we are allowed to use \eqref{Eq_Monomial_norm_1} and estimate this equation by
\begin{equation}\label{Eq_Monomial_norm_3}
\frac{|x|^\ell}{G_\ell}\exp\big(-\sigma|x|^{\varrho(|x|)}\big)\leq\frac{e^{\ell\sigma\frac{t}{\ell}}(e\rho)^{\frac{\ell}{\rho}}}{(e\sigma'\rho)^{\frac{\ell}{\rho}}}e^{-\sigma t}=\frac{1}{\sigma'^{\frac{\ell}{\rho}}},\qquad|x|\geq r_0,\,\ell\geq t_1.
\end{equation}
For $|x|\leq r_0$ on the other hand, we use \eqref{Eq_Monomial_norm_2} to get
\begin{equation}\label{Eq_Monomial_norm_4}
\frac{|x|^\ell}{G_\ell}\exp\big(-\sigma|x|^{\varrho(|x|)}\big)\leq\frac{r_0^\ell}{G_\ell}=\frac{r_0^\ell(e\rho)^{\frac{\ell}{\rho}}}{\varphi(\ell)^\ell}\leq\frac{1}{\sigma'^{\frac{\ell}{\rho}}},\qquad|x|\leq r_0,\ell\geq t_1.
\end{equation}
Combining now \eqref{Eq_Monomial_norm_3} and \eqref{Eq_Monomial_norm_4} gives the estimate
\begin{equation*}
\Vert x^\ell\Vert_{\varrho,\sigma}=\sup\limits_{x\in\mathbb{R}^{n+1}}|x|^\ell\exp\big(-\sigma|x|^{\varrho(|x|)}\big)\leq\frac{G_\ell}{\sigma'^{\frac{\ell}{\rho}}},\qquad\ell\geq t_1.
\end{equation*}
Finally, choosing the constant
\begin{equation*}
C(\sigma,\sigma'):=\max\bigg\{\max\limits_{0\leq\ell\leq t_1}\frac{\Vert x^\ell\Vert_{\varrho,\sigma}\sigma'^{\frac{\ell}{\rho}}}{G_\ell},\,1\,\bigg\},
\end{equation*}
gives the stated estimate \eqref{Eq_Monomial_norm} for every $\ell\in\mathbb{N}_0$.
\end{proof}

\begin{lemma}\label{lnova1}
There exists a constant $k$ depending only on $\rho$ for which the following statement holds: For any $\sigma>0$, we can take $C(\sigma)$ such that for any $f\in A_{\hat{\varrho},\sigma}$, and any $\ell\in\mathbb{N}_0$, the inequality
\begin{equation}\label{e10061056}
\frac{1}{\ell!}\Vert\partial^\ell_{x_0}f\Vert_{\hat{\varrho},k\sigma}\leq C(\sigma)\Vert f\Vert_{\hat{\varrho},\sigma}\frac{(2k \sigma)^{\ell/\rho}}{G_{\hat{\varrho},\ell}}
\end{equation}
holds.
\end{lemma}

\begin{proof}
Let $x\in\mathbb{R}^{n+1}$ and $j\in\mathbb{S}$, $u,v\in\mathbb{R}$ such that $x=u+jv$. then the Cauchy estimates for the restriction $f|_{\mathbb{C}_j}$ give for any $s>0$
\begin{align*}
|\partial^\ell_{x_0}f(x)|&\leq\frac{\ell!}{s^\ell}\max\limits_{|\xi|=s,\xi\in\mathbb{C}_j}|f(x+\xi)| \\
&\leq\frac{\ell!}{s^\ell}\Vert f\Vert_{\varrho,\sigma}\max\limits_{|\xi|=s,\xi\in\mathbb{C}_j}\exp\big(\sigma|x+\xi|^{\varrho(|x+\xi|)}\big) \\
&\leq\frac{\ell!}{s^\ell}\Vert f\Vert_{\varrho,\sigma}\exp\big(\sigma(|x|+s)^{\varrho(|x|+s)}\big) \\
&\leq\frac{\ell!}{s^\ell}\Vert f\Vert_{\varrho,\sigma}e^{\sigma C_\varepsilon}\exp\big(2^{\rho+\varepsilon}\sigma\big(r^{\varrho(r)}+s^{\varrho(s)}\big)\big),
\end{align*}
where in the last inequality we used Lemma \ref{lem_Properties_of_proximate_orders} i). Choosing now the special value $s=\varphi(\frac{\ell}{2^{\rho+\varepsilon}\sigma\rho})$, i.e $s^{\varrho(s)}=\frac{\ell}{2^{\rho+\varepsilon}\sigma\rho}$, we can further estimate
\begin{equation*}
|\partial^\ell_{x_0}f(x)|\leq\frac{\ell!}{\varphi(\frac{\ell}{2^{\rho+\varepsilon}\sigma\rho})^\ell}\Vert f\Vert_{\varrho,\sigma}e^{\sigma C_\varepsilon}\exp\big(2^{\rho+\varepsilon}\sigma r^{\varrho(r)}+\frac{\ell}{\rho}\big),\qquad\ell\in\mathbb{N}_0.
\end{equation*}
By Lemma \ref{l3005231} ii) there exists some $N\in\mathbb{N}$ such that
\begin{equation*}
\frac{\varphi(\ell)^\ell}{\varphi(\frac{\ell}{2^{\rho+\varepsilon}\sigma\rho})^\ell}\leq 2^{\frac{\ell}{\rho}}(2^{\rho+\varepsilon}\sigma\rho)^{\frac{\ell}{\rho}},\qquad\ell\geq N.
\end{equation*}
Hence we can choose a larger constant $C_1$ such that this inequality holds for all $\ell\in\mathbb{N}_0$, i.e.
\begin{equation*}
\frac{\varphi(\ell)^\ell}{\varphi(\frac{\ell}{2^{\rho+\varepsilon}\sigma\rho})^\ell}\leq C_1(2^{\rho+1+\varepsilon}\sigma\rho)^{\frac{\ell}{\rho}},\qquad\ell\in\mathbb{N}_0.
\end{equation*}
This leads to the estimate
\begin{align*}
|\partial^\ell_{x_0}f(x)|&\leq C_1\frac{\ell!}{\varphi(\ell)^\ell}\big((2^{\rho+1+\varepsilon}e\sigma\rho)^{\frac{1}{\rho}}\big)^\ell\Vert f\Vert_{\varrho,\sigma}e^{\sigma C_\varepsilon}\exp\big(2^{\rho+\varepsilon}\sigma r^{\varrho(r)}\big) \\
&=C_1e^{\sigma C_\varepsilon}\frac{\ell!}{G_\ell}(2^{\rho+1+\varepsilon}\sigma)^{\frac{\ell}{\rho}}\Vert f\Vert_{\varrho,\sigma}\exp\big(2^{\rho+\varepsilon}\sigma r^{\varrho(r)}\big), \qquad\ell\in\mathbb{N}_0.
\end{align*}
This is exactly the stated inequality \eqref{e10061056} for $\ell\geq N$.
\end{proof}

In view of the above stated properties, we can now  prove the following crucial results.

\begin{proposition}\label{psupersupernova1}
For an entire left slice monogenic function $f(x)$ belonging to $A_{\varrho,\sigma+0}$ for $\sigma\geq 0$, its Taylor expansion $\sum_{\ell=0}^\infty x^\ell a_\ell$ converges to $f(x)$ in the space $A_{\varrho,\sigma+0}$. In particular, the set of Fueter polynomials is dense in $A_{\varrho,+0}$ and also dense in $A_\varrho$.
\end{proposition}

\begin{proof}
For the former statement, it suffices to show that
\begin{equation*}
\sum_{\ell=0}^\infty\Vert x^\ell a_\ell\Vert_{\varrho, \sigma+\epsilon}
\end{equation*}
is finite for any $\epsilon>0$. By Lemma \ref{lsupersupernova1}, there exists a positive constant $C_0$ such that for any $\ell\geq 0$ we have
\begin{equation*}
\Vert x^\ell\Vert_{\varrho,\sigma+\epsilon}\leq C_0(\sigma+\epsilon/2)^{-\ell/\rho}G_{\hat{\varrho},\ell}.
\end{equation*}
On the other hand, by Remark \ref{r1006231857}, there exists a positive constant $C_1$ such that for any $\ell\geq 0$ we have
\begin{equation*}
|a_\ell|G_{\hat{\varrho},\ell}\leq C_1(\sigma+\epsilon/4)^{\ell/\rho}.
\end{equation*}
Therefore, we have
\begin{equation*}
\sum_{\ell=0}^\infty\Vert x^\ell a_\ell\Vert_{\varrho,\sigma+\epsilon}\leq C_0C_1\sum_{\ell=0}^\infty\Big(\frac{\sigma+\epsilon/4}{\sigma+\epsilon/2}\Big)^{\ell/\rho}<\infty.
\end{equation*}
For the latter statement in the case $f\in A_{\varrho,+0}$, it follows from the former one with $\sigma=0$ that
\begin{equation*}
\lim_{\ell\to\infty}\sum_{q\leq\ell}x^qa_q=f(x)
\end{equation*}
in the space $A_{\varrho,+0}$. In the case $f\in A_\varrho$, there exists $\sigma>0$ such that $f\in A_{\varrho,\sigma+0}$. Then the same convergence holds in the space $A_{\varrho, \sigma+0}$ and therefore also in the space $A_\varrho$.
\end{proof}

\begin{lemma}\label{l1106231654}
Let $\rho$ be a proximate order function, $\sigma,\tau>0$, $f\in A_{\rho,\sigma}$, $g\in A_{\rho,\tau}$. Then $f\star_Lg\in A_{\rho,\sigma+\tau}$ and
\begin{equation*}
\Vert f\star_L g\Vert_{\rho,\sigma+\tau}\leq 2^{\frac{n+4}{2}}\Vert f\Vert_{\rho,\sigma}\Vert g\Vert_{\rho,\tau},
\end{equation*}
where $n$ is the number of imaginary units in the Clifford algebra $\mathbb{R}_n$.
\end{lemma}

\begin{proof}
Due to Definition \ref{SHolDefMON}, the function $f\in\mathcal{S\!M}(\mathbb{R}^{n+1})$ does admit the decomposition
\begin{equation*}
f(u+jv)=f_0(u,v)+jf_1(u,v),\qquad u,v\in\mathbb{R},\,j\in\mathbb{S},
\end{equation*}
where for any arbitrary $j\in\mathbb{S}$, the functions $f_0,f_1$ are given by
\begin{equation*}
f_0(u,v)=\frac{f(u+jv)+f(u-jv)}{2}
\end{equation*}
and
\begin{equation*}
f_1(u,v)=J\frac{f(u-jv)-f(u+jv)}{2},
\end{equation*}
for any $u,v\in\mathbb{R}$. Hence, since $f\in A_{\rho,\sigma}$, the functions $f_0,f_1$ admit the estimates
\begin{align*}
|f_0(u,v)|&\leq\Vert f\Vert_{\rho,\sigma}\exp\Big(\sigma(u^2+v^2)^{\frac{1}{2}\rho\big((u^2+v^2)^{\frac{1}{2}}\big)}\Big), \\
|f_1(u,v)|&\leq\Vert f\Vert_{\rho,\sigma}\exp\Big(\sigma(u^2+v^2)^{\frac{1}{2}\rho\big((u^2+v^2)^{\frac{1}{2}}\big)}\Big).
\end{align*}
The same obviously holds true for the decomposition $g(u+jv)=g_0(u,v)+jg_1(u,v)$, i.e.
\begin{align*}
|g_0(u,v)|&\leq\Vert g\Vert_{\rho,\tau}\exp\Big(\tau(u^2+v^2)^{\frac{1}{2}\rho\big((u^2+v^2)^{\frac{1}{2}}\big)}\Big), \\
|g_1(u,v)|&\leq\Vert g\Vert_{\rho,\tau}\exp\Big(\tau(u^2+v^2)^{\frac{1}{2}\rho\big((u^2+v^2)^{\frac{1}{2}}\big)}\Big).
\end{align*}
Altogether, for every $x\in\mathbb{R}^{n+1}$ with the decomposition $x=u+jv$, we then get for the star product
\begin{align*}
&|(f\star g)(x)| \\
&=\Big|f_0(u,v)g_0(u,v)-f_1(u,v)g_1(u,v)+j\big(f_0(u,v)g_1(u,v)+f_1(u,v)g_0(u,v)\big)\Big| \\
&\leq 2^{\frac{n+4}{2}}\Vert f\Vert_{\rho,\sigma}\Vert g\Vert_{\rho,\tau}\exp\Big(\sigma(u^2+v^2)^{\frac{1}{2}\rho\big((u^2+v^2)^{\frac{1}{2}}\big)}\Big) \\
&\hspace{6cm}\times\exp\Big(\tau(u^2+v^2)^{\frac{1}{2}\rho\big((u^2+v^2)^{\frac{1}{2}}\big)}\Big) \\
&=2^{\frac{n+4}{2}}\Vert f\Vert_{\rho,\sigma}\Vert g\Vert_{\rho,\tau}\exp\Big((\sigma+\tau)|x|^{\rho(|x|)}\Big),
\end{align*}
where $x\in\mathbb{R}^{n+1}$ and the factor $2^{\frac{n}{2}}$ in the second lines comes from the fact that for two Clifford numbers $a,b\in\mathbb{R}_n$ the modulus of the product estimates as $|ab|\leq 2^{\frac{n}{2}}|a||b|$.
\end{proof}

The results of this section will be used to characterize continuous homomorphisms in terms of differential operators in the sense that will be specified in the next section.

\section{Differential operators, representations of continuous homomorphisms} \label{SEC3}

In this section we study continuous homomorphism from $A_{\varrho_1}$ and $A_{\varrho_2}$ and those from $A_{\varrho_1,+0}$ and $A_{\varrho_2,+0}$ where $\varrho_i(r)$ ($i=1,2$) are two proximate orders for positive orders $\rho_i=\lim_{r\to\infty}\varrho_i(r)>0$, satisfying
\begin{equation}\label{supernova}
r^{\varrho_1(r)}=O(r^{\varrho_2(r)}),\quad\text{as }r\to\infty.
\end{equation}

\begin{definition}\label{dnova1}
Let $\varrho_i$ ($i=1,2$) be two proximate orders for orders $\rho_i>0$ satisfying \eqref{supernova}. We take normalization $\hat{\varrho}_1$ of $\varrho_1$ as in Remark \ref{bem_normalization} and we define $G_{\hat{\varrho}_1,q}$ as in \eqref{Eq_Gl}. We denote by $\mathbf{D}_{\varrho_1\to\varrho_2}$ the set of all formal right linear differential operator $P$ of the form
\begin{equation*}
P=\sum_{\ell\in\mathbb{N}_0}u_\ell\star_L\partial_{x_0}^\ell
\end{equation*}
where the sequence $\{u_\ell\}_{\ell\in\mathbb N_0}\subset A_{\varrho_2}$ satisfies
\begin{equation}\label{enova1}
\forall\lambda>0,\,\exists\sigma>0,\,\exists C>0,\,\forall \ell\in\mathbb{N}_0,\,\Vert u_\ell\Vert_{\varrho_2,\sigma}\leq C\frac{G_{\hat{\varrho}_1,\ell}}{\ell!}\lambda^\ell.
\end{equation}
We denote by $\mathbf{D}_{\varrho_1\to\varrho_2,0}$ the set of all formal right linear differential operator $P$ of the form
\begin{equation*}
P=\sum_{\ell\in\mathbb{N}_0}u_\ell\star_L\partial_{x_0}^\ell
\end{equation*}
where the sequence $\{u_\ell\}_{\ell\in\mathbb N_0}\subset A_{\varrho_2}$ satisfies
\begin{equation}\label{enova2}
\forall\sigma>0,\,\exists\lambda>0,\,\exists C>0,\,\forall \ell\in\mathbb{N}_0,\,\Vert u_\ell\Vert_{\varrho_2,\sigma}\leq C\frac{G_{\hat{\varrho}_1,\ell}}{\ell!}\lambda^\ell.
\end{equation}
Note that in the latter case, each $u_\ell$ belongs to $A_{\varrho_2,+0}$.
\end{definition}

For the following remark see also \cite[Remark $4.4$]{AIO20}.

\begin{remark}\label{rnova1}
By adding $\ln(c)/\ln(r)$ for a constant $c>0$ (with a suitable modification near $r=0$) to a proximate order $\varrho_2(r)$ with order $\rho_2$, we get a new proximate order $\tilde{\varrho}_2(r)$ for the same order $\rho_2$ satisfying
\begin{equation*}
\tilde{\varrho}_2(r)=\varrho_2(r)+\ln(c)/\ln(r)
\end{equation*}
for $r>1$, that is,
\begin{equation*}
r^{\tilde{\varrho}_2(r)}=cr^{\varrho_2(r)},
\end{equation*}
eventually. Then $\Vert\cdot\Vert_{\tilde{\varrho}_2,\sigma}$ and $\Vert\cdot\Vert_{\varrho_2,c\sigma}$ become equivalent norms for $\sigma>0$, and the spaces $A_{\tilde{\varrho}_2}$ and $A_{\tilde{\varrho}_2,+0}$ are homeomorphic to $A_{\varrho_2}$ and $A_{\varrho_2,+0}$ respectively. By taking $c$ sufficiently large, we can take $\tilde{\varrho}_2$ as
\begin{equation*}
\varrho_1(r)\leq\tilde{\varrho}_2(r),\qquad r\geq r_0
\end{equation*}
for a suitable $r_0$, and we can choose normalizations $\hat{\varrho}_1$ of $\varrho_1$ and $\hat{\varrho}_2$ of $\tilde{\varrho}_2$ as

\begin{enumerate}
\item[(i)] $\hat{\varrho}_1(r)\leq\hat{\varrho}_2(r)$,\qquad $r\geq 0$.
\end{enumerate}

Since a proximate order and its normalization define equivalent norms, we have

\begin{enumerate}
\item[(ii)] $\Vert\cdot\Vert_{\hat{\varrho}_2,\sigma}$ and $\Vert\cdot\Vert_{\varrho_2,c\sigma}$ are equivalent for any $c>0$,
\item[(iii)] $\mathbf{D}_{\varrho_1\to\varrho_2}=\mathbf{D}_{\hat{\varrho}_1\to\hat{\varrho}_2}$, $\mathbf{D}_{\varrho_1\to\varrho_2,0}=\mathbf{D}_{\hat{\varrho}_1\to\hat{\varrho}_2,0}$.
\end{enumerate}

Note further that Theorems \ref{main1} and \ref{main2} below are not affected by the replacement of $\varrho_1$ and $\varrho_2$ by $\hat{\varrho}_1$ and $\hat{\varrho}_2$.
\end{remark}

\begin{theorem}\label{main1}
Let $\varrho_1,\varrho_2$ be two proximate orders satisfying \eqref{supernova}. Then all continuous linear operators $T:A_{\varrho_1}\to A_{\varrho_2}$ are characterized by operators of the form
\begin{equation}\label{Eq_T_rho}
T:=\sum\limits_{\ell=0}^\infty u_\ell\star_L\frac{\partial^\ell}{\partial x_0^\ell},
\end{equation}
with coefficients $\{u_\ell\}_{\ell\in\mathbb{N}}\subseteq\mathcal{S\!M}(\mathbb{R}^{n+1})$ satisfying
\begin{equation}\label{Eq_un_rho}
\forall\varepsilon>0,\,\exists\sigma,C>0: \Vert u_\ell\Vert_{\varrho_2,\sigma}\leq C\frac{G_{\varrho_1,\ell}}{\ell!}\varepsilon^\ell,
\end{equation}
i.e., $\mathbf{D}_{\varrho_1\to\varrho_2}$ is the set of all continuous operators from $A_{\varrho_1}$ to $A_{\varrho_2}$.
\end{theorem}

\begin{proof}
Due to Remark \ref{rnova1} we may assume that $\varrho_1(r)\leq\varrho_2(r)$ for every $r\geq 0$, which implies $A_{\varrho_1,\sigma}\subseteq A_{\varrho_2,\sigma}$ for every $\sigma>0$ and
\begin{equation*}
\Vert f\Vert_{\varrho_2,\sigma}\leq\Vert f\Vert_{\varrho_1,\sigma},\qquad f\in A_{\varrho_1,\sigma}.
\end{equation*}
\textit{Step 1}.  For the first implication, let $T$ be an operator of the form \eqref{Eq_T_rho} with coefficients $\{u_\ell\}_{\ell\in\mathbb{N}}$ satisfying \eqref{Eq_un_rho}. Let $\varepsilon>0$, then by \eqref{Eq_un_rho} there exists $\sigma,C>0$ such that
\begin{equation*}
\Vert u_\ell\Vert_{A_{\varrho_2,\sigma}}\leq C\frac{G_{\varrho_1,\ell}}{\ell!}\varepsilon^\ell,\qquad\ell\in\mathbb{N}.
\end{equation*}
Moreover, for $f\in A_{\varrho_1}$ one has $f\in A_{\varrho_1,\tau}$ for some $\tau>0$ and we can estimate
\begin{align*}
\Vert Tf\Vert_{\varrho_2,\sigma+2^{\rho_1+1}\tau}&\leq\sum\limits_{\ell=0}^\infty\Vert u_\ell\star_L\partial^\ell_{x_0}f\Vert_{\varrho_2,\sigma+2^{\rho_1+1}\tau} \\
&\leq 4\sum\limits_{\ell=0}^\infty\Vert u_\ell\Vert_{\varrho_2,\sigma}\Vert\partial^\ell_{x_0}f\Vert_{\varrho_2,2^{\rho_1+1}\tau} \\
&\leq 4C\sum\limits_{\ell=0}^\infty\frac{G_{\varrho_1,\ell}}{\ell!}\varepsilon^\ell\Vert\partial^\ell_{x_0}f\Vert_{\varrho_1,2^{\rho_1+1}\tau} \\
&\leq 4C\sum\limits_{\ell=0}^\infty\frac{G_{\varrho_1,\ell}}{\ell!}\varepsilon^\ell C(\tau)\frac{\ell!(2^{\rho_1+2}\tau)^{\frac{\ell}{\rho_1}}}{G_{\varrho_1,\ell}}\Vert f\Vert_{\varrho_1,\tau} \\
&=4CC(\tau)\Vert f\Vert_{\varrho_1,\tau}\sum\limits_{\ell=0}^\infty\big(\varepsilon(2^{\rho_1+2}\tau)^{\frac{1}{\varrho_1}}\big)^\ell.
\end{align*}
Choosing $\varepsilon<(2^{\rho_1+2}\tau)^{-\frac{1}{\varrho_1}}$, the right hand side is finite and we showed that $T$ is bounded as an operator from $A_{\varrho_1,\tau}$ to $A_{\varrho_2,\sigma+2^{\rho_1+1}\tau}$. Since $f$ was arbitrary, this also proves the continuity of $T$ as an operator from $A_{\rho_1} $ to $A_{\rho_2}$ according to the Definition \ref{defi_Arho0}.

\medskip

\textit{Step 2}. For the second implication, let $T:A_{\varrho_1}\to A_{\varrho_2}$ be an everywhere defined continuous operator. Then, thanks to the theory of locally convex spaces, there exists for every $\tau>0$ some $\tau'>0$, such that
\begin{equation}\label{fsupernova1}
\Vert Tf\Vert_{\varrho_2,\tau'}\leq C_\tau\Vert f\Vert_{\varrho_1,\tau},\qquad f\in A_{\varrho_1,\tau},
\end{equation}
(see \cite[Chap 4, Part 1, 5, Corollary 1]{G73}). Let us now define the functions
\begin{equation}\label{d1106231645}
u_\ell(x):=\frac{1}{\ell!}\sum_{k=0}^\ell{\ell\choose k}T(x^k)\star_L(-x)^{\ell-k}.
\end{equation}
For any $\varepsilon>0$ choose $\tau_0:=(\frac{2}{\varepsilon})^{\rho_1}$ and any $\tau>\tau_0$. Then the functions $u_\ell$ admit the estimate
\begin{align*}
\Vert u_\ell\Vert_{\varrho_2,\tau'+\tau }&\leq\frac{1}{\ell!}\sum\limits_{k=0}^\ell{\ell\choose k}\Vert T(x^k)\star_L(-x)^{\ell-k}\Vert_{\varrho_2,2^{\rho_1+1}(\tau'+\tau)} \\
&\leq\frac{2^{\frac{\ell+4}{2}}}{\ell!}\sum_{k=0}^\ell{\ell\choose k}\Vert T(x^k)\Vert_{\varrho_2,\tau'}\Vert x^{\ell-k}\Vert_{\varrho_2,\tau} \\
&\leq\frac{C(\tau)2^{\frac{\ell+4}{2}}}{\ell!}\sum_{k=0}^\ell{\ell\choose k}\Vert x^k\Vert_{\varrho_1,\tau}\Vert x^{\ell-k}\Vert_{\varrho_1,\tau} \\
&\leq\frac{C(\tau)C(\tau,\tau_0)2^{\frac{\ell+4}{2}}}{\ell!\tau_0^{\frac{\ell}{\varrho_1}}}\sum_{k=0}^\ell{\ell\choose k}G_kG_{\ell-k} \\
&\leq\frac{C(\tau)C(\tau,\tau_1)2^{\frac{\ell+4}{2}}G_\ell}{n!\tau_0^{\frac{\ell}{\varrho_1}}}2^\ell=\frac{C(\tau) C(\tau,\tau_1)2^{\frac{\ell+4}{2}}G_\ell}{\ell!}\varepsilon^\ell.
\end{align*}
This estimate shows that the functions $u_\ell$ satisfy the assumptions \eqref{Eq_un_rho}. Next consider a function $f(x)=\sum\limits_{m=0}^Mx^ma_m$ be a polynomial. Then
\begin{align}
\sum_{\ell=0}^\infty u_\ell(x)\star_L\partial_{x_0}^\ell f(x)&=\sum_{\ell=0}^\infty\sum\limits_{m=0}^\infty u_\ell(x)\star_L\partial_{x_0}^\ell x^ma_m \notag \\
&=\sum_{\ell=0}^\infty\sum\limits_{m=\ell}^\infty\sum_{k=0}^\ell{\ell\choose k}{m\choose\ell}T(x^k)\star_L(-x)^{\ell-k}\star_Lx^{m-\ell}a_m \notag \\
&=\sum_{\ell=0}^\infty\sum\limits_{m=\ell}^\infty\sum_{k=0}^\ell(-1)^{\ell-k}{n\choose k}{m\choose \ell}T(x^k)\star_Lx^{m-k}a_m \notag \\
&=\sum\limits_{m=0}^\infty\sum\limits_{k=0}^m\sum_{\ell=0}^{m-k}(-1)^\ell{m-k\choose \ell}\,{m\choose k}T(x^k)\star_Lx^{m-k}a_m \notag \\
&=\sum\limits_{m=0}^\infty\sum\limits_{k=0}^m\delta_{m-k,0}{m\choose k}T(x^k)\star_Lx^{m-k}a_m \notag \\
&=\sum\limits_{m=0}^\infty T(x^m)a_m=T\left(\sum\limits_{m=0}^\infty x^ma_m\right)=Pf(x). \label{Eq_P_representation}
\end{align}
Since the polynomials are dense by Lemma \ref{psupersupernova1} and the operator $T$ is continuous, the same result holds for any function $f\in A_{\varrho_1}$.
\end{proof}

\begin{theorem}\label{main2}
Let $\varrho_1,\varrho_2$ be two proximate orders satisfying \eqref{supernova}. Then all continuous linear operators $T:A_{\varrho_1,+0}\to A_{\varrho_2,+0}$ are characterized by operators of the form
\begin{equation}\label{Eq_T_rhoplus}
T:=\sum\limits_{\ell=0}^\infty u_\ell \star_L\frac{\partial^\ell}{\partial x_0^\ell},
\end{equation}
with coefficients $\{ u_\ell\}_{\ell\in\mathbb N} \subseteq\mathcal{S\!M}(\mathbb{R}^{n+1})$ satisfying
\begin{equation}\label{Eq_enova2}
\forall\sigma>0,\,\exists\varepsilon,C>0: \Vert u_\ell\Vert_{\varrho_2,\sigma}\leq C\frac{G_{\varrho_1,\ell}}{\ell!}\varepsilon^\ell,
\end{equation}
i.e., $\mathbf{D}_{\varrho_1\to\varrho_2,0}$ is the set of all continuous operators from $A_{\varrho_1,+0}$ to $A_{\varrho_2,+0}$.
\end{theorem}

\begin{proof}
Due to Remark \ref{rnova1} we may assume that $\varrho_1(r)\leq\varrho_2(r)$ for every $r\geq 0$, which implies $A_{\varrho_1,\sigma}\subseteq A_{\varrho_2,\sigma}$ for every $\sigma>0$ and
\begin{equation*}
\Vert f\Vert_{\varrho_2,\sigma}\leq\Vert f\Vert_{\varrho_1,\sigma},\qquad f\in A_{\varrho_1,\tau}.
\end{equation*}
We proceed in two steps.

\textit{Step 1}.  For the first implication let $T$ be an operator of the form \eqref{Eq_T_rhoplus} with coefficients $\{u_\ell\}_{\ell\in\mathbb{N}}$ satisfying \eqref{Eq_enova2}. Let $\sigma>0$. Then with the coefficients $\varepsilon,C>0$ from \eqref{Eq_enova2} we choose $0<\tau<\frac{1}{2^{\rho_1+2}\varepsilon^{\rho_1}}$. Then we obtain the estimate
\begin{align*}
\sum_{\ell\in\mathbb N_0}\Vert u_\ell\star_L\partial^\ell_{x_0}f\Vert_{\varrho_2,\sigma+2^{\rho_1+1}\tau}&\leq 2^{\frac{n+4}{2}}\sum_{\ell\in\mathbb N_0}\Vert u_\ell\Vert_{\varrho_2,\sigma}\Vert\partial^\ell_{x_0}f\Vert_{\varrho_2,2^{\rho_1+1}\tau} \\
&\leq 2^{\frac{n+4}{2}}C(\eta,\sigma,\tau)\sum_{\ell\in\mathbb N_0}\Vert u_\ell\Vert_{\varrho_2,\sigma}\Vert\partial^\ell_{x_0}f\Vert_{\varrho_1,2^{\rho_1+1}\tau} \\
&\leq C'(\eta,\sigma,\tau)\sum_{\ell\in\mathbb N_0} \frac{G_{\varrho_1,\ell}\varepsilon^\ell}{\ell!}\Vert f\Vert_{\varrho_1,\tau}\frac{\ell!}{G_{\rho_1,\ell}}(2^{\rho_1+2}\tau)^{\frac{\ell}{\rho_1}} \\
&=C'(\eta,\sigma,\tau)\Vert f\Vert_{\varrho_1,\tau}\sum_{\ell\in\mathbb N_0}\varepsilon^\ell(2^{\rho_1+2}\tau)^{\frac{l}{\rho_1}},
\end{align*}
for $f\in A_{\varrho_1,\tau}$. Note that the last sum is finite due to the choice of $\tau$. This implies that $P:A_{\varrho_1,+0}\to A_{\varrho_2,+0}$ is a continuous operator.

\medskip

\textit{Step 2}. For the second implication let $P:A_{\varrho_1,+0}\to A_{\varrho_2,+0}$ be a continuous operator. Then for every $\varepsilon>0$ there exist $\tau>0$ and $C (\varepsilon)>0$, such that $P:A_{\varrho_1,\tau}\to A_{\varrho_2,\varepsilon}$ is continuous with
\begin{equation}\label{e1106231704}
\Vert Pf\Vert_{\varrho_2,\varepsilon}\leq C(\varepsilon) \Vert f\Vert_{\varrho_1,\tau},\qquad f\in A_{\varrho_1,\tau},
\end{equation}
(see \cite[Chap 4, Part 1, 5, Corollary 1]{G73}). Defining the functions $u_\ell$ as in \eqref{d1106231645}, they admit the estimate
\begin{equation}\label{enovissima3}
\begin{split}
&\Vert u_\ell\Vert_{\varrho_2,\sigma}\leq 2^{\frac{n+4}{2}} \sum_{k\leq\ell}\frac{\Vert x^{\ell-k}\Vert_{\varrho_2,\sigma/2}\Vert P(x^k)\Vert_{\varrho_2,\sigma/2}}{(\ell-k)!k!} \\
&\leq\sum_{k\leq\ell}\frac{\Vert x^{\ell-k}\Vert_{\varrho_1,\sigma/2}C(\tau_1)\Vert x^k\Vert_{\varrho_1,\tau_1}}{(\ell- k)!k!} \\
&\leq\sum_{k\leq\ell}\frac{C(\tau_0,\sigma'/2)\tau_0^{-(\ell-k)/\rho_1}G_{\varrho_1,\ell-k}C(\tau_1)C(\tau_0,\tau_1)\tau_0^{-k/\rho_1}G_{\varrho_1,k}}{(\ell- k)!k!} \\
&\leq C(\tau_1)C(\tau_0,\sigma/2)C(\tau_0,\tau_1)\sum_{k\leq\ell}{\ell\choose k}\frac{G_{\varrho_1,\ell}}{\ell!}\tau_0^{-\ell/\rho_1} \\
&=C(\tau_1)C(\tau_0,\sigma/2)C(\tau_0,\tau_1)\frac{G_{\varrho_1,\ell}}{\ell!}2^\ell\tau_0^{-\ell/\rho_1}.
\end{split}
\end{equation}
Here we used Lemma \ref{l1106231654} at the first inequality, \eqref{e1106231704} at the second with $\varepsilon=\sigma/2$, Lemma \ref{lsupersupernova1} twice with $0<\tau_0<\sigma/2$ and $0<\tau_0<\tau_1$ at the third, and Lemma \ref{lsupersupernova2} at the fourth. Therefore, by defining
\begin{equation*}
\lambda:=2\tau_0^{-1/\rho_1},
\end{equation*}
we have
\begin{equation*}
\Vert u_\ell\Vert_{\varrho_2,\sigma}\leq C'\frac{G_{\varrho_1,\ell}}{\ell!}\lambda^\ell
\end{equation*}
for any $\ell$. Since $\sigma>0$ can be chosen arbitrarily, the functions $u_\ell$ satisfy all the conditions in \eqref{Eq_enova2}. Moreover, the operator $P$ admits the representation
\begin{equation*}
P=\sum_{\ell=0}^\infty u_\ell\star_L\partial_{x_0}^\ell
\end{equation*}
in the same way as in \eqref{Eq_P_representation} for polynomials. Since the polynomials are dense due to Lemma \ref{psupersupernova1} and the operator $T$ is continuous, the same result holds for any function $f\in A_{\varrho_1,+0}$. Thus $P\in\mathbf{D}_{\varrho_1\to\varrho_2,0}$.
\end{proof}

\end{document}